\documentclass{amsart}
\usepackage{amsmath,amsfonts,amsthm,amssymb,stmaryrd,comment}

\renewcommand{\leq}{\leqslant}
\renewcommand{\geq}{\geqslant}

\newcommand{\ind}[1]{1_{#1}}

\newcommand{\bbz}{\mathbb{Z}}
\newcommand{\bbr}{\mathbb{R}}
\newcommand{\bbq}{\mathbb{Q}}
\newcommand{\bbc}{\mathbb{C}}

\newcommand{\td}{\,\mathrm{d}}
\newcommand*{\bbe}{
  \mathop{
    \mathchoice{\vcenter{\hbox{\larger[4]$\mathbb{E}$}}}
               {\kern0pt\mathbb{E}}
               {\kern0pt\mathbb{E}}
               {\kern0pt\mathbb{E}}
  }\displaylimits
}

\newcommand{\abs}[1]{\left\lvert #1\right\rvert}
\newcommand{\Abs}[1]{\lvert #1\rvert}
\newcommand{\brac}[1]{\left( #1\right)}

\newcommand{\norm}[1]{\left\lVert #1\right\rVert}

\newcommand{\Norm}[1]{\lVert #1\rVert}

\newtheorem{theorem}{Theorem}
\newtheorem{lemma}{Lemma}

\newtheorem{corollary}{Corollary}
\newtheorem{conjecture}{Conjecture}
\newtheorem{proposition}{Proposition}

\theoremstyle{definition}

\begin{document}
\title{Integers with small digits in multiple bases}
\author{Thomas F. Bloom}
\address{Department of Mathematics, University of Manchester, Manchester, M13 9PL}
\email{thomas.bloom@manchester.ac.uk}
\author{Ernie Croot}
\address{Department of Mathematics, Georgia Tech, Atlanta, USA}
\maketitle

\begin{abstract}
We show that, for any $r\geq 1$, if $g_1,\ldots,g_r$ are distinct coprime integers, sufficiently large depending only on $r$, then for any $\epsilon>0$ there are infinitely many integers $n$ such that all but $\epsilon \log n$ of the digits of $n$ are $\leq g_i/2$ in base $g_i$ for all $1\leq i\leq r$. In other words, for any fixed large bases, there are infinitely many $n$ such that almost all of the digits of $n$ are small in all bases simultaneously. This is both a quantitative and qualitative improvement over previous work of Croot, Mousavi, and Schmidt. As a consequence, we obtain a weak answer to a conjecture of Graham concerning divisibility of $\binom{2n}{n}$. 
\end{abstract}
Graham conjectured (see, for example, \cite{EG}) that there are infinitely many integers $n$ such that $\binom{2n}{n}$ is coprime to $105=3\cdot 5\cdot 7$. By Kummer's criterion, a prime $p$ does not divide $\binom{2n}{n}$ if and only if all of the base $p$ digits of $n$ are $< p/2$. Graham's conjecture is therefore equivalent to the question of whether there are infinitely many integers such that in base $3$ all digits are in $\{0,1\}$, in base $5$ all digits are in $\{0,1,2\}$, and in base $7$ all digits are in $\{0,1,2,3\}$. For example,
\[756 = (1001000)_3 = (11011)_5 = (2130)_7.\]

This is a special case of a more general phenomenon, where we expect infinitely many $n$ all of whose digits are small in multiple bases simultaneously, provided such bases are sufficiently large. To get an idea how large is necessary, we have the following heuristic of Pomerance \cite{Po}: note first that, if $N$ is a power of $g$, then the number of $0\leq n<N$ all of whose base $g$ digits are $< g/2$ is exactly $N^{\log_g(\lceil g/2\rceil)}$. Heuristically, therefore, the chance that $n$ has all of its base $g$ digits $< g/2$ is $\approx n^{\log_g(\lceil g/2\rceil)-1}$. Assuming these events are independent for distinct bases $g_1,\ldots,g_r$, we expect that the chance that $n\in[N,2N)$ has all base $g_j$ digits $\leq g_j/2$ (for all $1\leq j\leq r$) is $\approx N^{\sum_j\log_{g_j}(\lceil g_j/2\rceil)-r}$. If this is $>N^{-1}$ then it follows that there should exist at least one such $n\in [N,2N)$ for all sufficiently large $N$. Therefore, provided the bases $g_j$ are `independent' and
\[\sum_j\log_{g_j}(\lceil g_j/2\rceil)-r=-\sum_j \log_{g_j}\brac{\frac{g_j}{\lceil g_j/2\rceil}}>-1,\]
there should exist infinitely many $n$ all of whose digits base $g_j$ are $< g_j/2$ for all $1\leq j\leq r$ simultaneously.

One obvious way that the digits between different bases $g_j$ may not be independent is when one base is a power of another base. To rule out such dependencies, it is natural to insist that $g_i^{a_i}\neq g_j^{b_j}$ for all $i\neq j$ and integers $a_i,b_j\geq 1$ (notice for example that this is certainly satisfied if the bases are pairwise coprime). Assuming this condition, and generalising the heuristic of Pomerance above, we are led to the following conjecture. 
\begin{conjecture}\label{conj}
Let $r\geq 1$ and $g_1,\ldots,g_r\geq 2$ be integers such that $g_i^{a_i}\neq g_j^{b_j}$ for all $i\neq j$ and integers $a_i,b_j\geq 1$, with associated $\kappa_1,\ldots,\kappa_r\in(0,1]$ such that 
\begin{equation}\label{eq-conjcond}\sum_{1\leq j\leq r}\log_{g_j}\brac{\frac{g_j}{\lceil \kappa_jg_j\rceil}}<1.
\end{equation}
There are infinitely many $n$ such that, for all $1\leq j\leq r$, all base $g_j$ digits of $n$ are $< \kappa_j g_j$.
\end{conjecture}
For example, Graham's conjecture is the case when $g_1=3$, $g_2=5$, $g_3=7$, and $\kappa_1=\kappa_2=\kappa_3=1/2$, and indeed we calculate
\[\log_3(3/2)+\log_5(5/3)+\log_7(7/4)=0.974\cdots < 1.\]

The main result of this paper establishes a weak version of this conjecture, weaker in two ways: we have a stricter condition (although of a similar shape) on how large the bases need to be, and (more significantly) we only obtain almost all of the digits being small. 
\begin{theorem}\label{th-ologn}
Let $r\geq 1$ and $g_1,\ldots,g_r\geq 2$ be integers such that $g_i^{a_i}\neq g_j^{b_j}$ for all $i\neq j$ and integers $a_i,b_j\geq 1$, with associated $\kappa_1,\ldots,\kappa_r\in(0,1]$ such that 
\begin{equation}\label{eq-const}
\sum_{1\leq j\leq r} \log_{g_j} (320r^5/\kappa_j)< \frac{1}{2r}.\end{equation}
(Note in particular that this holds for all $g_1,\ldots,g_r$ sufficiently large in terms of $r,\kappa_1,\ldots,\kappa_r$). 

For any $\epsilon>0$ there are infinitely many $n$ such that, for all $1\leq j\leq r$, all but at most $\epsilon \log n$ of the base $g_j$ digits of $n$ are $< \kappa_j g_j$. 
\end{theorem}
Indeed, the proof actually shows that such an $n$ exists in any interval of the shape 
\[[N,\exp(O(\epsilon^{-1-o(1)}))N]\]
assuming $N$ is sufficiently large depending on $\epsilon,g_1,\ldots,g_r,\kappa_1,\ldots,\kappa_r$ (where the implicit constant depends on $g_1,\ldots,g_r,\kappa_1,\ldots,\kappa_r$ only). We stress, however, that Theorem \ref{th-ologn} is ineffective, in the sense that, given $\epsilon,g_1,...,g_r,\kappa_1,...,\kappa_r$, it does not quantify how large the smallest $n$ must be; by the above we can say that there exist 
\[\gg \epsilon^{1+o(1)}\log N\]
many such integers $n\leq N$ when $N$ is sufficiently large (depending on $\epsilon$), but Theorem~\ref{th-ologn} gives no control on how large this needs to be. 

Put another way, one could ask for a more quantitative result that allows for $\epsilon=\epsilon(n)\to 0$ as $n\to \infty$. A more technically complicated version of the argument in this paper should allow
\[\epsilon=\epsilon(n)\ll (\log n)^{-c}\]
for some constant $c>0$ (depending on the $g_i$ and $\kappa_i$). The limit of the method in this paper, even assuming a strong quantitative form of Schanuel's conjecture, appears to be $\epsilon=(\log n)^{-1/2+o(1)}$. To avoid complicating an already technically involved proof, in this paper we will focus on establishing the qualitative statement of Theorem~\ref{th-ologn}, and defer further discussion about quantitative issues to Section~\ref{sec-imp}.

Kummer's theorem implies that, for any prime $p$, if $p^k$ is the highest power of $p$ that divides $\binom{2n}{n}$ then $k$ is the number of base $p$ digits of $n$ that are $\geq p/2$. A trivial consequence is that, for any fixed primes $p_1,\ldots,p_r$, for any integer $n$ we can write $\binom{2n}{n}=n_1n_2$ where $(n_1,p_1\cdots p_r)=1$ and $n_2\leq n^{O(1)}$.

Theorem~\ref{th-ologn} immediately implies the following improvement, which can be viewed as a weak version of Graham's conjecture.
\begin{corollary}\label{cor-graham}
Suppose $r\geq 3$ and $p_1,\ldots,p_r$ are primes, all sufficiently large depending on $r$ only. For any $\epsilon>0$ there are infinitely many $n$ such that
\[\binom{2n}{n}=n_1n_2\]
where $(n_1,p_1\cdots p_r)=1$ and $n_2\leq n^\epsilon$. 
\end{corollary}
The quantitative improvement to Theorem~\ref{th-ologn} mentioned earlier would allow us to improve this to
\[n_2 \leq \exp((\log n)^{1-c})\]
for some constant $c>0$.
\subsection*{On constants}
The particular shape of the condition \eqref{eq-const} should not be taken too seriously; while we have tried to produce an explicit bound, we have also compromised a little in the strength of this bound to make the proof easier to follow. To get an idea of what this threshold means in concrete terms, one can calculate that, when $r=3$ and $\kappa_j=1/2$, the condition \eqref{eq-const} is satisfied provided $g_1,g_2,g_3\geq 10^{94}$. In particular, Corollary~\ref{cor-graham} holds for any three distinct prime bases $p_1,p_2,p_3\geq 10^{94}$. 

If one more carefully optimised the constants in this argument (at the cost of making the proofs and statements cosmetically messier) then a back of the envelope calculation suggests this threshold could be brought down to perhaps an order of magnitude closer to $\approx 10^{15}$, but achieving a more reasonable threshold (say $<100$) probably requires some new ideas.
\subsection*{Previous work}

A weaker form of Theorem~\ref{th-ologn} was obtained in earlier work of the second author, Mousavi, and Schmidt \cite{CMS}, who proved that, for any $\epsilon>0$ and $r\geq 1$, if $p_1,\ldots,p_r$ are primes sufficiently large depending on both $\epsilon$ and $r$, then there are infinitely many $n$ such that all except $\epsilon \log n$ of the base $p_j$ digits of $n$ are $< p_j/2$ for all $1\leq j\leq r$. In particular, note that $p_1,\ldots,p_r$ must be taken sufficiently large depending on $\epsilon$, not just $r$.

The argument of this paper is a strengthening of that of \cite{CMS}, with the aim of removing this dependence on $\epsilon$. Many component lemmas of \cite{CMS} have been improved both quantitatively and qualitatively; another significant change is that we avoid discretising and rather than work over $\mathbb{Z}/N\mathbb{Z}$ perform our Fourier analysis over $\mathbb{R}/\mathbb{Z}$. Finally, instead of appealing to Parseval's identity as in \cite{CMS}, we introduce a new estimate, Lemma~\ref{lem:spec}, for the number of large Fourier coefficients of sets with small digits.

The case $r=2$ was earlier studied by Erd\H{o}s, Graham, Ruzsa, and Straus \cite{EGRS}, who proved a much stronger result.

\begin{theorem}[Erd\H{o}s, Graham, Ruzsa, and Straus \cite{EGRS}]\label{th-egrs}
If $g_1,g_2\geq 2$ are integers and $1/g_1\leq \kappa_1\leq 1$ and $1/g_2\leq \kappa_2\leq 1$ are such that
\begin{equation}\label{eq-cond}\frac{\lceil \kappa_1g_1\rceil-1}{g_1-1}+\frac{\lceil\kappa_2g_2\rceil-1}{g_2-1}\geq 1\end{equation}
then there are infinitely many $n$ whose base $g_i$ digits are $<\kappa_i g_i$ for $i=1,2$.
\end{theorem}

In particular, for any distinct primes $p_1,p_2$ there exist infinitely many $n$ such that $\binom{2n}{n}$ is coprime to $p_1p_2$, resolving the $r=2$ analogue of Graham's conjecture. Theorem~\ref{th-egrs} is almost the final word on the case $r=2$, except that the assumption \eqref{eq-cond} is stronger than the conjectured sufficient condition \eqref{eq-conjcond}, which here is equivalent to
\[\frac{\log(\lceil \kappa_1g_1\rceil)}{\log g_1}+\frac{\log(\lceil \kappa_2g_2\rceil)}{\log g_2}>1.\]
(The condition \eqref{eq-cond} is indeed stronger since $\frac{\log x}{\log g}>\frac{x-1}{g-1}$ for $1<x< g$.) It would be of interest even for $r=2$ to prove that condition \eqref{eq-conjcond} is sufficient.

As a concrete challenge, can one prove the existence of infinitely many $n$ whose digits are all $\in \{0,1\}$ in both bases $3$ and $5$ simultaneously? Conjecture~\ref{conj} predicts that this is true, since $\log_3(3/2)+\log_5(5/2)=0.938\cdots$, but this is not covered by Theorem~\ref{th-egrs}, since $\frac{1}{2}+\frac{1}{4}=0.75<1$.

\subsection*{Acknowledgements} Thomas Bloom is supported by a Royal Society University Research Fellowship.

\section*{Structure of the paper}
In Section~\ref{sec-sketch} we will give an overview of the proof of Theorem~\ref{th-ologn}, in particular motivating the form of the main technical proposition which will be used in the semi-constructive proof. In Section~\ref{sec-tech} we will state this proposition precisely as Proposition~\ref{prop-main}, and formally deduce Theorem~\ref{th-ologn}. 

The proof of Proposition~\ref{prop-main} will occupy the rest of the paper. In Section~\ref{sec-fourier} we prove it using Fourier analysis, assuming further lemmas on the separation of power sums (proved in Section~\ref{sec-sep}) and the behaviour of large exponential sums over integers with small digits (proved in Section~\ref{sec-spec}). 

The reader generally interested in additive problems concerning integers with small digits may like to jump immediately to Section~\ref{sec-spec} and peruse the lemma and proof therein, which we expect will be of wider use.

Finally, in Section~\ref{sec-imp} we will sketch how to prove a quantitative form of Theorem~\ref{th-ologn}, and discuss some connections with Schanuel's conjecture.

\section*{Notation and conventions}
For $x\in \bbr$ we write $\norm{x}\in (-1/2,1/2]$ for the distance from $x$ to the nearest integer. We also write $\lfloor x\rfloor$ for the greatest integer $n\leq x$ and $\{x\}=x-\lfloor x\rfloor$ for the fractional part of $x$.

We use the vector notation $\vec{k}=(k_1,\ldots,k_d)$ (the dimension $d$ will always be clear from context). We will use this notation for both elements of $\bbr^d$, $\bbz^d$, and $(\bbr/\bbz)^d$.

We identify $\bbr/\bbz$ with $[0,1)$ where convenient. If $\psi\in L^1(\bbr/\bbz)$ then we define the Fourier transform $\widehat{\psi}:\bbz\to \bbc$ by
\[\widehat{\psi}(k)=\int_0^1 \psi(t)e(-kt)\td t,\]
where $e(x)=e^{2\pi ix}$. By a slight abuse of notation when $A\subset (\bbr/\bbz)^d$ is a finite set we also write
\[\widehat{1_A}(\vec{k})=\sum_{\vec{a}\in A}e(-\vec{a}\cdot \vec{k})\]
for $\vec{k}\in \bbz^d$.

\section{Sketch of the proof}\label{sec-sketch}

We will prove Theorem~\ref{th-ologn} in a semi-constructive way. The idea is to begin with some candidate number and successively alter blocks of digits in each base to attempt to make the digits smaller, while avoiding affecting too many of the digits in other bases, or the digits in higher blocks in that same base.

This is an extension of the method used by Erd\H{o}s, Graham, Ruzsa, and Straus \cite{EGRS} in the case $r=2$ (see Theorem \ref{th-egrs}) and of Croot, Mousavi, and Schmidt \cite{CMS}. 

To illustrate the method of Erd\H{o}s, Graham, Ruzsa, and Straus when $r=2$, we will work with the bases $g_1 = 3$ and $g_2 = 5$, and choose $\kappa_1 = \kappa_2 = 1/2$.  In other words, we seek integers $n$ such that each of the base $3$ digits are $0$ or $1$ and each of the base $5$ digits are $0,1,$ or $2$.  

As mentioned above, we adopt a constructive approach, starting with a number like $3^{12}=(1000000000000)_3$, which has the form we seek in base $3$.  Unfortunately, in base $5$ this number is $(114001231)_5$ (that is, 
$3^{12} = 5^8 + 5^7 + 4\cdot 5^6 + 5^3 + 2\cdot 5^2 
+ 3\cdot 5 + 1$).  Working from most to least significant digits, the first problematic digit here is the $4$ (contributing the $4 \cdot 5^6$).  By adding a lower power of $3$ to our initial $3^{12}$, we can make this digit smaller, at the expense of making the digit $1$ corresponding to the $5^7$ into a $2$.  Specifically, since $3^9=(1112213)_5$, adding $3^9$ to produce $3^{12} + 3^9$ we get
\begin{align*}
(1001000000000)_3
&=(114001231)_5+(001112213)_5\\
&=(120113444)_5.
\end{align*}
Continuing in this manner, adding smaller and smaller powers of $3$ to clear out large digits in the base $5$ expansions of those numbers, we arrive at the number $(1001000101110)_3$, which has base $5$ expansion $(120121111)_5$. It is not immediately obvious, of course, that such a procedure is always possible, i.e. that one can remove a large base $5$ digit without creating large base $3$ digits. That this is always possible is proved using elementary number theory in \cite{EGRS} (and makes essential use of the condition \eqref{eq-cond}).

When $r=2$ (and condition \eqref{eq-cond} is satisfied) this constructive method always works. When dealing with three or more bases simultaneously, however, problems quickly arise: when fixing a problematic large digit in one base there are now at least two other bases to worry about, and the elementary counting argument of \cite{EGRS} no longer suffices to ensure that all the bases can be controlled simultaneously.

In the present work, instead of starting with some $g_i^N$, which has base $g_i$ expansion $(1000\cdots 0)_{g_i}$, and then adding progressively lower and lower order (than $N$) base $g_i$ digits so as to make the higher-order digits in the other bases small, we work in some large `universal' base $L = \ell^h$, for some $h \geq 1$, where gcd$(\ell, g_1g_2\cdots g_r)=1$.  (There is no other restriction on $\ell$, and in particular one can simply take $\ell=2$ when all $g_i$ are odd.)  

We will work relative to this large base $L$, and choose each new base $L$ digit that we add to force long blocks of digits in the bases $g_1,...,g_r$ to be small simultaneously. A precise description of what is required is given in Proposition \ref{prop-main}:  when choosing the base $L$ digit corresponding to $L^k$ we can ensure that all the digits $a_i$ in a base $g_j$ can be made to satisfy $a_i < \kappa_j g_j$ so long as $g_j^i \in [C L^k, C^{-1} L^{k+1}]$.  The parameter $C$ here will be `small' compared to $L$, so that we can get almost all of the base $g_j$ digits corresponding to a single base $L$ digit to be $< \kappa_j g_j$. A significant complication arises from the fact that at the point we need to choose any base $L$ digit all higher base $L$ digits have already been fixed. This means that when choosing the current base $L$ digit, with the goal of fixing the corresponding block of base $g_j$ digits to be small, we are not starting from all zeros in base $g_j$ but from some unknown `noise', the result of the choices made for higher $L$ digits.

This unknown noise is represented by the parameter $\beta$ in Proposition~\ref{prop-main}, and it is vital that this proposition should hold uniformly for all $\beta\in \bbr$. Roughly speaking then, Proposition~\ref{prop-main}, the main technical result of this paper, states that (assuming various conditions), for almost all $n$, given any $\beta\in \bbr$, we can shift $\beta$ by some quantity $O(HL^{k})$ (where $H$ is some parameter slightly larger than $L$) such that almost all of the corresponding block of base $g_j$ digits in $[CL^k,C^{-1}L^{k+1}]$ are small, for all bases $g_j$ simultaneously. (Importantly, since we are only altering $\beta$ by something $O(HL^{k})$ almost all of the higher digits of $\beta$ in every base $g_j$ are unaffected, and hence once they are small they remain small throughout the process.)

It is worth mentioning where the exceptional set of at most $\varepsilon \log n$ digits in Theorem \ref{th-ologn} in each of the bases comes from.  There are two sources, both coming from Proposition \ref{prop-main}.  The first source is the fact that we have the interval $[C L^k, C^{-1} L^{k+1}]$ instead of an interval $[L^k, L^k H]$, say, where that extra `padding' factor $C$ means we lose control over the digits with $g_i^j \in [L^k, C L^k)$ and $(C^{-1}L^{k+1}, L^k H]$ -- the number of digits we cannot control in this way is 
\[\ll\frac{\log(CH/L)}{\log L}\log n.\]

The other main source of potential bad digits comes from the fact that we cannot ensure a suitable modification is possible for all base $L$ digits, but only almost all. The bound on the number of possible bad $L$ digits, where we cannot control anything, given in Proposition~\ref{prop-main} is
\[\ll (\log L)^{O(1)}(L/H)^{1/r}\log n,\]
each such bad block contributing $O(\log L)$ many bad base $g_i$ digits.

The choices of $L$ and $H$ are therefore a little delicate; to ensure that the number of bad digits is at most $\epsilon \log n$ in total we choose $H=L^{1+O(\epsilon)}$ and $L=\epsilon^{-O(\epsilon^{-1}r)}$.

Finally, we conclude by discussing the proof of Proposition \ref{prop-main}. After dilating the scale under consideration so that we are working in $[0,1]$, and writing $\vec{x}\in (\bbr/\bbz)^r$ for the rescaled $L^k$ (with the direction $i$ corresponding to the behaviour in base $g_i$) what we need to know is essentially whether a fixed $\vec{\beta}\in (\bbr/\bbz)^r$ can be shifted by some constant multiple of $\vec{x}$ to lie in $A+I$, where $A\subseteq [0,1]^r$ is a finite subset, with the component in the $i$ direction corresponding to those numbers with small digits in base $g_i$, and $I$ is some small box around the origin. 

This is proved by an application of Fourier analysis in $(\bbr/\bbz)^r$ -- the point is that the required statement is `generically true', and the only exceptional $\vec{x}$ must have some unusual correlations with large Fourier coefficients of the set $A$. By bounding the number of such exceptions we can control the number of potential `bad' $\vec{x}$, and hence the number of potential bad base $L$ digits, and complete the proof of Proposition~\ref{prop-main}.

 Bounding the number of such exceptional $\vec{x}$ requires a number of auxiliary lemmas, some of which have variants that appeared in \cite{CMS} (although the versions in the present work are much more flexible and general).  
 
 The main new ingredient that allows us to get stronger bounds than this previous work is Lemma \ref{lem:spec}, which gives an upper bound on the number of places where a certain exponential sum over integers with small base $g$ digits is larger in magnitude than a given fraction 
$\eta$ of its maximum possible value (where this maximum value is $|A|$) -- that is, we have an exponential sum over numbers of the form $\sum_{0 \leq i < R} c_i g^i$, where $0 \leq c_i < t$. (For comparison the analysis in \cite{CMS} bounded the number of such large exponential sums by an application of Parseval's identity, and did not exploit the specific structure of the set $A$.)  These bounds are proved using the fact that this exponential sum can be written as a Riesz product over geometric progressions. Analyzing the integers $k$ where this product can be large in terms of the base $g$ digits of that $k$, it is found that all but a small number of these digits must lie inside some small set.  Bounding the number of such patterns of base $g$ digits of these numbers $k$ then gives the lemma.

Establishing control on exponential sums over sets of integers with restricted digits has also played a significant role in a number of other problems; we mention in particular the work of Fouvry and Mauduit \cite{FoMa05} on the additive behaviour of sets of integers omitting specified digits, and the work of Maynard \cite{ma, granville} on the existence of primes omitting specified digits. To the best of our knowledge, however, a result like Lemma \ref{lem:spec} which controls the exponential sum over the set of integers with only \emph{small} digits has not before appeared in the literature.

\section{Main technical result}\label{sec-tech}
We will now state our main technical result, and show how it implies Theorem~\ref{th-ologn}. The proof of Proposition~\ref{prop-main} will take up the remainder of the paper.

\begin{proposition}\label{prop-main}
Let $r\geq 1$ and $g_1,\ldots,g_r\geq 2$ be integers such that $g_i^{a_i}\neq g_j^{b_j}$ for all $i\neq j$ and integers $a_i,b_j\geq 1$, with associated $\kappa_1,\ldots,\kappa_r\in(0,1]$ such that 
\[
\sum_{1\leq j\leq r} \log_{g_j}\brac{\frac{10g_j}{\lceil \tfrac{\kappa_j}{32r^5} g_j\rceil}}< \frac{1}{2r}.\]
Let $\ell\geq 2$ satisfy $(\ell,g_1\cdots g_k)=1$. There exists a constant $C>0$ (depending on $\ell,g_j,\kappa_j$) such that the following holds.

Let $H\geq L\geq C$, and suppose that $L$ is a power of $\ell$. If $N$ is sufficiently large (depending on $\ell,g_j,\kappa_j,H,L$) then, for all except at most $(\log L)^C(L/H)^{1/r}N$ many $n\leq N$, for any $\beta\in \bbr$, there exists an integer $s_{n,\beta}\in [1,H]$ such that, for all $1\leq j\leq r$, if 
\[s_{n,\beta}L^n+\beta=\sum_k a_kg_j^k\quad\textrm{ with }\quad 0\leq a_k<g_j\]
then
\[0\leq a_k< \kappa_j g_j\quad\textrm{ for }\quad g_j^k\in [CL^n,\tfrac{1}{C}L^{n+1}].\]
\end{proposition}

This statement is quite technical and difficult to parse. The point of the conclusion is that, for almost all `blocks' of digits of length $\approx \log L$, any $\beta$ can be shifted by some integer of scale $O(HL^n)$ to ensure that all of its base $g_j$ digits in that block (except the fringes) are small, simultaneously for all $j$. Furthermore since we are shifting by an integer of size $O(HL^n)$ very few of the higher digits will be affected. 

We can then work our way downwards along the blocks, beginning with the highest digits, repeatedly shifting to ensure that all of the digits except those in the fringes of each block (and perhaps all digits in any of the exceptional blocks) are small in all bases simultaneously.

Note that we need to take $L$ smaller than $H$ by at least $(\log L)^{Cr}$ or else the conclusion is trivial. To get an idea about the scale of parameters, for our application we will choose $\ell\geq 2$ to be minimal such that $(\ell,g_1\cdots g_k)=1$, and then choose $H$ to be some large constant depending only on $\epsilon$ and $L\approx H^{1-O(\epsilon)}$. With this choice of parameters the number of exceptional blocks $n\leq N$ is $\ll H^{-O(\epsilon/r)}N\leq \epsilon N$ for suitable choice of $H$.

The proof of Proposition~\ref{prop-main} will take the majority of the paper. In this section we will show how it implies Theorem~\ref{th-ologn}. 

\begin{proof}[Proof of Theorem~\ref{th-ologn}]
Choose $\ell$ minimal so that $(\ell,g_1\cdots g_k)=1$. We regard $\ell, g_1,\ldots,g_r$ and $\kappa_1,\ldots,\kappa_r$ as fixed -- in particular any implicit constant may depend on them. Let $\epsilon>0$. Let $H\geq L$ be chosen later (depending on $\epsilon$) but large enough so that the conclusion of Proposition~\ref{prop-main} holds, and let $N$ be sufficiently large (depending on $\epsilon,H,L$). 

It suffices to construct some integer $b\in [L^N,HL^{N+1}]$ such that, for $1\leq j\leq r$, all but at most $\epsilon \log b$ of the base $g_j$ integers of $b$ are $< \kappa_jg_j$. We will construct $b$ inductively in a `top-down' fashion, by inductively choosing $s_N,\ldots,s_{0}\in [0,H]$. It is convenient to introduce the notation
\[b_{\geq d}=\sum_{d\leq i\leq N}s_iL^{i}.\]
Our final $b$ will then be $b=b_{\geq 0}$. Let $s_N=1$ and, for $0\leq n< N$, 
\begin{enumerate}
\item if the conclusion of Proposition~\ref{prop-main} fails for $n$ then $s_n=0$, and
\item if the conclusion of Proposition~\ref{prop-main} holds for $n$ then let $1\leq s_n \leq H$ be the minimal integer thus provided for $\beta=b_{\geq n+1}$, so that for $1\leq j\leq r$ if 
\[b_{\geq n}=s_nL^{n}+b_{\geq n+1}=\sum_k a_kg_j^{k}\quad\textrm{ with }\quad 0\leq a_k<g_j\]
then 
\[0\leq a_k< \kappa_j g_j\quad\textrm{ for }\quad g_j^k\in [CL^n,\tfrac{1}{C}L^{n+1}].\]
\end{enumerate}
For brevity, let $D^\flat$ denote those $0\leq n<N$ where the first alternative holds and $D^\sharp$ those $n$ where the second holds. In particular, note that $\lvert D^\flat\rvert \ll (\log L)^C(L/H)^{1/r}N$.

We first note that, as required, 
\[L^{N}\leq b \leq HL^{N+1},\]
and hence in particular $N=\log_Lb+O_{H,L}(1)$. We now consider the intervals 
\[I_{n}=[C\tfrac{H}{L}L^{n}, \tfrac{1}{C}L^{n+1}]\]
for $0\leq n<N$. Note that if $g_j^{u_j}$ is the highest power of $g_j$ which appears in the base $g_j$ expansion of $b$ then
\[g_j^{u_j}\leq b\leq HL^{N+1}.\]
Therefore if $g_j^u$ corresponds to a non-zero digit in the base $g_j$ expansion of $b$ then either $g_j^u\in I_{n}$ for some $0\leq n<N$, or
\[g_j^u\in [1,C\tfrac{H}{L})\cup (\tfrac{1}{C}L^{N},HL^{N+1}]\cup
\bigcup_{1\leq n< N}(\tfrac{1}{C}L^{n},C\tfrac{H}{L}L^{n}].\]
The number of digits in the base $g_j$ expansion of $b$ corresponding to $g_j^u\not\in \cup I_n$ is therefore at most
\[\ll \log(CH/L)N+O_{H,L}(1)\ll \frac{\log(CH/L)}{\log L}\log b+O_{H,L}(1).\]

Furthermore, since each $I_n$ contains at most $O(\log L)$ many $g_j^u$, the number of digits corresponding to $g_j^u\in I_n$ for some $n\in D^\flat$ is at most
\[\ll (\log L)^{C+1}(L/H)^{1/r}N\ll (\log L)^{C}(L/H)^{1/r}\log b.\]
In total, therefore, the number of digits which do \emph{not} correspond to  some $g_j^u\in I_n$ for some $n\in D^\sharp$ is 
\[\ll \brac{\frac{\log (CH/L)}{\log L}+(\log L)^{C+1}(L/H)^{1/r}}\log b+O_{H,L}(1).\]
If we choose $L=\lfloor H^{1-C'\epsilon}\rfloor$ and $H=\lceil \epsilon^{-C'r\epsilon^{-1}}\rceil$ for some large constant $C'>0$, and take $N$ sufficiently large, then this is at most $O(\epsilon \log b)$.

It remains to prove that the digits corresponding to $g_j^u\in I_n$ for $n\in D^\sharp$ are all $<\kappa_j g_j$. By construction if 

\[b_{\geq n}=\sum_k a_kg_j^{k}\quad\textrm{ with }\quad 0\leq a_k<g_j\]
and 
\[g_j^k\in [CL^{n},\tfrac{1}{C}L^{n+1}]\supseteq I_n\]
then $a_k< \kappa_jg_j$. Now note that
\[b-b_{\geq n}=\sum_{0\leq i\leq n-1}s_iL^i\leq H\frac{L^{n}}{L-1}\]
and the general fact (e.g. provable by induction) that, for any $\ell\geq 1$, if $r< g_j^\ell$ and $s$ is such that the digit corresponding to $g_j^{\ell}$ in $s$ is $\leq g_j-2$, then the digits corresponding to $g_j^k$ are identical in $s$ and $r+s$ for all $k>\ell$. In particular, since we can assume that $\kappa_jg_j\leq g_j-2$, the digits corresponding to $g_j^k$ in $b$ and $b_{\geq n}$ are identical provided
\[g_j^k> g_j\tfrac{H}{L-1}L^n.\]
We can assume that $C$ is large enough such that the right-hand side is $\leq C\tfrac{H}{L}L^n$, and hence the digits corresponding to $g_j^k\in I_n$ are identical in both $b$ and $b_{\leq n}$, and hence are also $<\kappa_jg_j$ in $b$, as required.
\end{proof}

\section{An application of Fourier analysis}\label{sec-fourier}
Proposition~\ref{prop-main} will be proved by the circle method -- that is, using Fourier analysis and a careful study of where the relevant Fourier transforms can be large. In particular, we will use Fourier analysis to prove the following technical result and then explain how it implies Proposition~\ref{prop-main}, assuming further lemmas which will be proved in later sections.

\begin{proposition}\label{prop-fourier}
Let $I_1,\ldots,I_r\subset \bbr/\bbz$ be intervals of length $\delta_1,\ldots,\delta_r$ respectively, and write $I=I_1\times \cdots \times I_r$ and $\delta=\prod_{i=1}^r\delta_i$. Let $A\subset (\bbr/\bbz)^r$ be a finite set such that $\vec{a}-\vec{b}\not\in \prod_i[-\delta_i/2,\delta_i/2]$ for all $\vec{a}\neq \vec{b}\in A$, and write $\alpha=\delta\abs{A}$. Let $h\geq 3$ be an integer.

For any $H\geq 1$ and $\vec{x}\in(\bbr/\bbz)^r$, either
\begin{enumerate}
\item for any $\vec{\beta}\in (\bbr/\bbz)^r$ there exists $1\leq s\leq H$ such that\footnote{Here $hA$ denotes the $h$-fold iterated sumset of $A$, that is, all $x\in (\bbr/\bbz)^r$ of the shape $a_1+\cdots+a_h$ for some $a_i\in A$.} $s\vec{x}+\vec{\beta}\in hA+I$, or
\item there exists some $\vec{k}\in \bbz^r\backslash \{0\}$ such that

\begin{enumerate}
\item $\Abs{k_i}\ll_r (\log \frac{1}{\delta})^2\delta_i^{-1}$ for $1\leq i\leq r$,
\item \[\Abs{\widehat{\ind{A}}(\vec{k})}\gg_r \Norm{\vec{k}}_\infty^{-\frac{r+1}{h}}\abs{A},\]
and
\item $\Norm{\vec{k}\cdot \vec{x}}\ll_r \log(1/\alpha)H^{-1}$.
\end{enumerate}
\end{enumerate}
\end{proposition}
\begin{proof}
Let $\ell \geq 1$ be determined later, and let $H'=\lfloor H/\ell\rfloor$. We fix $\vec{x}\in (\bbr/\bbz)^r$ and suppose that the first alternative fails, so that there exists some $\vec{\beta}\in(\bbr/\bbz)^r$ such that there is no  $1\leq s\leq H$ such that $s\vec{x}+\vec{\beta}\in hA+I$. Without loss of generality (shifting $\vec{\beta}$ if necessary) we can assume $I_i=[-\delta_i/2,\delta_i/2]$ for $1\leq i\leq r$. Let $J\geq 1$ be some integer parameter to be chosen later, and let $\psi_i:\bbr/\bbz\to \bbr$ be a function supported on $I_i$ such that
\begin{enumerate}
\item $\widehat{\psi_i}(0)=\| \psi_i\|_1=1$,
\item $0\leq \widehat{\psi_i}(k)\leq \min(1,(J^2/\delta_i k)^{2J})$ for all $k\in \mathbb{Z}$, and
\item $\sum_{k\in \bbz} \widehat{\psi_i}(k)=\psi_i(0)\leq 4/\delta_i$.
\end{enumerate}
(The existence of such a $\psi_i$ is a standard construction in Fourier analysis; details are given in Appendix~\ref{app}.) We write $\psi(\vec{x})=\psi_1(x_1)\cdots\psi_r(x_r)$. The basic theory of the Fourier transform implies that $\psi_i(x_i)=\sum_{k\in \bbz}\widehat{\psi_i}(k)e(kx_i)$ and hence
\[\psi(\vec{x}) = \sum_{\vec{k}\in \bbz^r}\widehat{\psi_1}(k_1)\cdots\widehat{\psi_r}(k_r)e(\vec{k}\cdot\vec{x})=\sum_{\vec{k}\in \bbz^r}\widehat{\psi}(\vec{k})e(\vec{k}\cdot\vec{x}),\]
say. Since the first alternative fails, and $\psi$ is supported on $I$, we know that $\psi(s\vec{x}+\vec{\beta}-\vec{a_1}-\cdots-\vec{a_h})=0$ for all $\vec{a_1},\ldots,\vec{a_h}\in A$ and $1\leq s\leq H$, and therefore
\[0=\sum_{\vec{a_1},\ldots,\vec{a_h}\in A}\sum_{1\leq s_1,\ldots,s_\ell\leq H'}\psi((s_1+\cdots+s_\ell)\vec{x}+\vec{\beta}-\vec{a_1}-\cdots-\vec{a_h}).\]
Inserting the above Fourier series for $\psi$ and rearranging, recalling that $\widehat{\ind{A}}(\vec{k})=\sum_{\vec{a}\in A}e(-\vec{k}\cdot\vec{a})$,
\[0=\sum_{\vec{k}\in \bbz^r}\widehat{\psi}(\vec{k})\widehat{\ind{A}}(\vec{k})^h\sum_{1\leq s_1,\ldots,s_\ell\leq H'}e(\vec{k}\cdot ((s_1+\cdots+s_\ell)\vec{x}+\vec{\beta})).\]
The term $\vec{k}=(0,\ldots,0)$ contributes $\abs{A}^h(H')^\ell$ to this sum. By the triangle inequality, therefore,  
\[\abs{A}^h(H')^\ell\leq \sum_{\vec{k}\in \bbz^r\backslash 0}\Abs{\widehat{\psi}(\vec{k})}\Abs{\widehat{\ind{A}}(\vec{k})}^h\Abs{f(\vec{k})}^\ell\]
where $f(\vec{k})=\sum_{1\leq s\leq H'}e(s(\vec{k}\cdot \vec{x}))$. We have, for any $M$ and fixed $1\leq i\leq r$,
\begin{align*}
\sum_{\substack{\vec{k}\in \bbz^r\backslash 0\\ \Abs{k_i} \geq M}}\Abs{\widehat{\psi}(\vec{k})}\Abs{\widehat{\ind{A}}(\vec{k})}^h\Abs{f(\vec{k})}^\ell
&\leq 
\abs{A}^h(H')^\ell\sum_{\substack{\vec{k}\in \bbz^r\backslash 0\\ \Abs{k_i} \geq M}}\Abs{\widehat{\psi}(\vec{k})}\\
&\leq \frac{1}{2}4^{r}\delta_i\delta^{-1}\abs{A}^h(H')^\ell\sum_{k\geq M}(J^2/\delta_i k)^{2J}.
\end{align*}
The sum here is, using the bound $\sum_{k\geq M}k^{-2J}\leq 2J(2/M)^{2J-1}$,
\[(J^2/\delta_i)^{2J}\sum_{k\geq M}k^{-2J}\leq 2J^3\delta_i^{-1}(2J^2/\delta_iM)^{2J-1}.\]
If we take $M=8J^2\delta_i^{-1}$ then this is at most $8J^3\delta_i^{-1}4^{-2J}$. In particular, if we choose $J=C\lceil r+\log\delta^{-1}\rceil$ for some large absolute constant $C$ then the total contribution from $\vec{k}$ with $\lvert k_i\rvert \geq 8J^2\delta_i^{-1}$ for some $1\leq i\leq r$ is at most $\frac{1}{2}\abs{A}^h(H')^{\ell}$. 

Furthermore, since $\vec{a}-\vec{b}\not\in I$ for $\vec{a}\neq \vec{b}\in A$,
\begin{align*}
\sum_{\vec{k}}\Abs{\widehat{\psi}(\vec{k})}\Abs{\widehat{\ind{A}}(\vec{k})}^h
&\leq \abs{A}^{h-2}\sum_{\vec{a},\vec{b}\in A}\sum_{\vec{k}} \widehat{\psi}(\vec{k})e(\vec{k}\cdot (\vec{a}-\vec{b}))\\
&= \abs{A}^{h-2}\sum_{\vec{a},\vec{b}\in A}\psi(\vec{a}-\vec{b})\\
&\leq 4^r\delta^{-1}\abs{A}^{h-1}=4^r\alpha^{-1}\abs{A}^h,
\end{align*}
and hence since $\Abs{f(\vec{k})}\leq \min(H',1/\Norm{\vec{k}\cdot \vec{x}})$, the contribution from those $\vec{k}$ such that $\Norm{\vec{k}\cdot \vec{x}}^\ell\geq 4^{r+1}\alpha^{-1}/(H')^\ell$ is at most $\frac{1}{4}\abs{A}^h(H')^\ell$. 

In particular, if $\ell=\lceil 2r+\log(1/\alpha)\rceil$, say, then we can discard the contribution from those $\vec{k}$ such that $\| \vec{k}\cdot\vec{x}\| \geq 4/H'=4/\lfloor H/\ell\rfloor$. 

Therefore, if $\Gamma'$ is the set of all $\vec{k}\neq 0$ such that $\Abs{k_i}\leq 8J^2\delta_i^{-1}$ for all $1\leq i\leq r$ and $\Norm{\vec{k}\cdot \vec{x}}\leq 4/\lfloor H/\ell\rfloor$ then 
\[\tfrac{1}{4}\abs{A}^h\leq \sum_{\vec{k}\in \Gamma}\Abs{\widehat{\psi}(\vec{k})}\Abs{\widehat{\ind{A}}(\vec{k})}^h.\]

Finally, for any $K\geq 1$ and $c>0$, the contribution to this sum from $\vec{k}$ with $\Norm{\vec{k}}_\infty\in [K/2,K)$ such that 
\[\Abs{\widehat{\ind{A}}(\vec{k})}^h\leq cK^{-r-1}\abs{A}^h\]
is (since there are at most $(2K)^r$ such $\vec{k}$) at most
\[\brac{cK^{-r-1}\abs{A}^h}(2K)^r=c2^rK^{-1}\abs{A}^h.\]
Choosing $c=2^{-r-5}$, say, and summing over dyadic values of $K$, we conclude that the contribution from those $\vec{k}$ with $\lvert \widehat{1_A}(\vec{k})\rvert^h \leq 2^{-2r-6}\|\vec{k}\|_\infty^{-r-1}\abs{A}^h$ is at most $\frac{1}{16}\abs{A}^h$. The second conclusion now follows, since the remaining sum is $\geq \frac{1}{16}\abs{A}^h$ and hence in particular must sum over a non-empty set of $\vec{k}$. 
\end{proof}

We may now prove Proposition~\ref{prop-main} (assuming further lemmas to be proved in subsequent sections). 
\begin{proof}[Proof of Proposition~\ref{prop-main}]
We fix some integers $g_1<\cdots<g_r$ with associated $\kappa_1,\ldots,\kappa_r$ and $\ell$ as in the statement of the proposition. All implied constants in this proof may depend on $r,g_1,\ldots,g_r,\kappa_1,\ldots,\kappa_r$.

Let $C>0$ be some constant to be chosen later, but depending only on $\ell,\kappa_1,\ldots,\kappa_r,g_1,\ldots,g_r$. Let $H\geq L\geq C$ (where $L$ is a power of $\ell$) and for $1\leq j\leq r$ let $R_j\geq 1$ be chosen such that
\[g_j^{R_j}\leq L<g_j^{R_j+1}\]
and, more generally, for $n\geq 1$ and $1\leq j\leq r$ let $\ell_j=\ell_j(n)=\lfloor n/\log_L g_j\rfloor$, so that 
\[g_j^{\ell_j}\leq L^n<g_j^{\ell_j+1}.\]
Consider $\vec{x}(n)\in [0,1)^r$ defined by 
\[x_j(n)=L^ng_j^{-\ell_j-R_j}= g_j^{\{n/\log_Lg_j\}-R_j}\in [g_j^{-R_j},g_j^{-R_j+1}).\]
Let $h=h(r)\geq 3$ be some parameter to be chosen later (which will depend only on $r$) and 
\[A_j=\left\{ \sum_{1\leq i\leq R_j}c_ig_j^{-i} : 0\leq c_i<\tfrac{\kappa_j}{h}g_j\textrm{ for }1\leq i\leq R_j\right\}\subset [0,1).\]
Let $\delta_j=g_j^{-R_j-1}$, so that distinct elements of $A_j$ are separated by $>\delta_j$, and $I_j=[0,\delta_j)$. 
It suffices to show that (if $N$ is sufficiently large) for all but at most $(\log L)^{O(1)}(L/H)^{1/r} N$ many $n\leq N$, for all $\vec{\beta}\in \bbr^r$ there exists some $1\leq s\leq H$ such that, for all $1\leq j\leq r$,
\[sx_j(n)+\beta_j\in hA_j+I_j+\bbz.\]
Indeed, if we apply this with $\beta_j=\beta g_j^{-\ell_j-R_j}$ then this means that the base $g_j$ digits of $(sL^n+\beta)g_j^{-\ell_j-R_j}$ corresponding to $g_j^{-u}\in [g_j^{-R_j},1)$ are all $<\kappa_jg_j$, and we obtain the conclusion after multiplying by $g_j^{\ell_j+R_j}$ (noting that $g^{-R_j}\gg 1/L$).

Let $A=A_1\times \cdots \times A_r$ and $I=[0,\delta_1)\times \cdots \times [0,\delta_r)$.
Note that $\abs{A}=\prod_j\lceil\tfrac{\kappa_j}{h}g_j\rceil^{R_j}$, so that with
\[\delta=\prod \delta_j=\prod_{1\leq j\leq r}g_j^{-R_j-1}\]
we have
\[\alpha = \delta \abs{A}\geq (ch)^{-O(R_1+\cdots+R_r)}\prod_{j} \kappa_j^{R_j}\]
for some constant $c\ll 1$. Note that $\log(1/\delta)\ll \log L$. Let $\Gamma\subset \bbz^r$ be the set of all $\vec{k}\in \bbz^r$ such that 

\begin{enumerate}
\item $\Abs{k_i}\leq C_1(\log \frac{1}{\delta})^2\delta_i^{-1}$ for all $1\leq i\leq r$ and
\item \[\Abs{\widehat{\ind{A}}(\vec{k})}\geq c_2 \Norm{\vec{k}}_\infty^{-\frac{r+1}{h}}\abs{A},\]
\end{enumerate}
where the constants $C_1,c_2$ (which will depend on $r$ only) are those in the second part of the conclusion of Proposition~\ref{prop-fourier}.

We say that $n\leq N$ is bad for $\vec{k}$ if $\Norm{\vec{k}\cdot \vec{x}(n)}\leq C_3\log(1/\alpha)/H$ (where again $C_3$ is an appropriate constant as delivered by the conclusion of Proposition~\ref{prop-fourier}). Invoking Proposition~\ref{prop-fourier} it suffices to prove that the number of $n\leq N$ which are bad for some $\vec{k}\in \Gamma$ is at most $(\log L)^{O(1)}(L/H)^{1/r} N$.

Let $\Gamma_M\subseteq \Gamma$ be the set of $\vec{k}\in \Gamma$ with $\Norm{\vec{k}}_\infty\in [M/2,M)$. We will bound, for each fixed $\vec{k}\in \Gamma_M$, the number of $n$ which are bad for $\vec{k}$ -- that is, we need to count the number of $n\leq N$ such that
\[\left\| \sum_j k_j g_j^{-R_j} g_j^{\{\frac{n}{\log_L g_j}\}}\right\|\leq \epsilon,\]
where $\epsilon=C_3\log(1/\alpha)/H$. The idea is that the exponents $\{\frac{n}{\log_Lg_j}\}$ should behave generically like independent random elements from $[0,1]$ and each coefficient has size $\asymp M/L$, and so this inequality should be true with `probability' $\ll (\epsilon L/M)^{1/r}$.

Making this intuition precise is the content of Section~\ref{sec-sep} (which is significantly complicated by the fact that, without Schanuel's conjecture, we do not know that the $\frac{1}{\log_Lg_j}$ are independent over $\mathbb{Q}$). In particular, applying Lemma~\ref{lem:badbound} of Section~\ref{sec-sep} (with $M/L\ll \max_i\abs{\zeta_i}\ll (\log L)^2$) implies that, for sufficiently large $N$, for each $\vec{k}\in \Gamma_M$, the number of $1\leq n\leq N$ which are bad for $\vec{k}$ is 
\[\ll (\log L)^2\brac{\log(1/\alpha)L/HM}^{1/r}N.\]

We will now bound the size of $\Gamma_M$ itself. We know that if $\vec{k}\in \Gamma_M$ then $\| \vec{k}\|_\infty <M$ and  
\[\prod_{i=1}^r \abs{\sum_{a\in A_j}e(k_ja)}\gg M^{-\frac{r+1}{h}}\prod_{j=1}^r\abs{A_j},\]
and hence in particular for all $1\leq j\leq r$
\[\abs{\sum_{a\in A_j}e(k_ja)}\gg M^{-2r/h} \abs{A_j}.\]
We will use this lower bound on the size of the exponential sum over $A_j$ to restrict the number of possible choices for each $k_j$. Giving an upper bound on the number of such $k_j$ is the goal of Section~\ref{sec-spec}. In particular, Lemma~\ref{lem:spec} of that section implies that, since $M\ll (\log L)^2L$ and $L\asymp g_j^{R_j}$, the number of such $0\leq k_j< M$ is 
\[\ll (\log L)^{O(1)} M^{\log_{g_j}(10g_j/t_j)+\sqrt{8r/h}},\]
say, where $t_j=\lceil \frac{\kappa_j}{h}g_j\rceil$. It follows that 
\[\abs{\Gamma_M}\ll (\log L)^{O(1)} M^{\sum_j\log_{g_j}(10g_j/t_j)+\sqrt{8r^3/h}}.\]
Combining our two estimates we deduce that the number of $n$ which are bad for some $\vec{k}\in \Gamma_M$ is 
\[\ll (\log L)^{O(1)}(L/H)^{1/r}\log(1/\alpha)M^{\sum_j\log_{g_j}(10g_j/t_j)+\sqrt{8r^3/h}-1/r}N.\]
We now choose $h=32r^5$. Since $\log(1/\alpha)\ll R_1+\cdots+R_r\ll \log L$, this is 
\[\ll (\log L)^{O(1)}(L/H)^{1/r}M^{\sum_j \log_{g_j}(10g_j/t_j)-1/2r}N,\]
and by assumption 
\[\sum_j \log_{g_j}(10g_j/t_j)-1/2r< 0.\]
Summing over $M=2^k$ for all $k\geq 0$ implies that the number of $n$ which are bad for any $\vec{k}\in \Gamma$ is 
\[\ll (\log L)^{O(1)}(L/H)^{1/r}N\]
as required.
\end{proof}
\section{Separation}\label{sec-sep}
In this section we study the following problem: given some $\zeta_j,\theta_j$ can we bound the number of $n\leq N$ such that
\[\| \zeta_1\theta_1^{\{\frac{n}{\log_Lg_1}\}}+\cdots \zeta_r\theta_r^{\{\frac{n}{\log_Lg_r}\}}\| \leq \epsilon?\]
Suppose that we knew $1,\frac{1}{\log_Lg_1},\ldots,\frac{1}{\log_Lg_r}$ to be linearly independent over $\mathbb{Q}$. An application of the classical Weyl's theorem then implies that the vectors
\[\left(\left\{\frac{n}{\log_L g_1}\right\},\ldots,\left\{\frac{n}{\log_L g_r}\right\}\right)\]
are uniformly distributed in $[0,1)^r$ as $n\to \infty$. In particular it suffices to bound, for some fixed parameter $\delta$, the number of $\delta$-separated $\vec{t}\in [0,1)^r$ such that 
\[\| \zeta_1\theta_1^{t_1}+\cdots+\zeta_r\theta_r^{t_r}\| \leq \epsilon.\]
By counting along 1-dimensional `slices' this in turn can be reduced to the problem of bounding the number of $t\in [0,1)$ such that
\[\| \zeta_1\theta_1^{t}+\cdots+\zeta_r\theta_r^{t}\| \leq \epsilon.\]
(Note the coefficients $\zeta_i$ may be differ as we vary over the 1-dimensional slices, and hence some reasonable uniformity in the dependence on $\zeta_i$ is required.)
Finally, we may give an upper bound on the number of such $t$ by showing that in any large $\delta$-separated sequence of $t$ there must be some value where $\sum \zeta_j \theta_j^t$ is large, which in turn is proved by an elementary argument using induction on $r$ and the mean value theorem.

There is, however, a complication in that it is not know how to prove that the $\frac{1}{\log_L g_j}$ are linearly independent over $\bbq$ (although this is an immediate consequence of Schanuel's conjecture). Surprisingly, however, this does not cause significant difficulties: we can simply take any possible linear dependencies into account, and apply the above argument to a maximal linearly independent subset of the $\frac{1}{\log_L g_j}$, and apply Weyl's theorem in a suitable lower-dimensional space. 

We will now put this plan into action. The first step is the following auxiliary lemma.
\begin{lemma}\label{lem-lb}
Let $x_1,\ldots,x_r>0$ be distinct real numbers with associated $c_1,\ldots,c_r\in \bbr$. If $f(t)=\sum_{1\leq i \leq r}c_ix_i^{t}$ then for any $0<v_1<\cdots<v_{2^{r-1}}<1$
\[\max_{1\leq j\leq 2^{r-1}}\abs{f(v_j)}\gg_{x_1,\ldots,x_r}\Delta^{r-1}\max_j\abs{c_j},\]
where $\Delta=\min\abs{v_{k+1}-v_k}$.
\end{lemma}
\begin{proof}
We use induction on $r$. The case $r=1$ is trivial. In general, suppose that $r\geq 2$ and the statement holds for $r-1$. Without loss of generality we can assume that $\abs{c_1}\leq \abs{c_i}$ for all $2\leq i\leq r$. Consider
\[g(t) = x_1^{-t}f(t)=c_1+\sum_{2\leq k\leq r}c_k(x_k/x_1)^t.\]
Note that since $\abs{f(t)}\geq \min(1,x_1)\abs{g(t)}$ for all $t\in[0,1]$ it suffices to prove 
\[\max_{0\leq j\leq 2^{r-1}}\abs{g(v_j)}\gg_{x_1,\ldots,x_r}
\Delta^{r-1}\max_{2\leq j\leq r}\abs{c_j}.\]

By the mean value theorem there is some $w_j\in [v_{2j-1},v_{2j}]$ for $1\leq j\leq 2^{r-2}$ such that 
\[\max(\abs{g(v_{2j})},\abs{g(v_{2j-1})})\geq \abs{g(v_{2j})-g(v_{2j-1})}\geq \Delta \Abs{g'(w_j)}.\]
Note that $\min_{1\leq j\leq 2^{r-2}}\Abs{w_{j+1}-w_j}\geq \Delta$ so we are done by the inductive hypothesis applied to 
\[g'(t)=\sum_{2\leq j\leq r}(\log x_j/x_1)c_j(x_j/x_1)^t.\]
\end{proof}
We may now solve our analytic problem in the simplest one-dimensional case.
\begin{lemma}\label{lem:curve1}
Let $\epsilon>0$ and $\zeta_1,\ldots,\zeta_r\in \bbr$. Let $\theta_1,\ldots,\theta_r>0$ be distinct real numbers (all $\neq 1$) and $\delta>0$. Suppose that $\abs{\zeta_1\theta_1^t+\cdots+\zeta_r\theta_r^t}\leq Z$ for all $t\in [0,1)$.

The number of $\delta$-separated $t\in[0,1)$ such that
\[\norm{\zeta_1\theta_1^t+\cdots+\zeta_r\theta_r^t}\leq \epsilon\]
is 
\[O_{\theta_1,\ldots,\theta_r,r}(Z\delta^{-1}\epsilon^{1/r}(\max_i \abs{\zeta_i})^{-1/r}).\]
\end{lemma}
\begin{proof}
It suffices to show that, for any integer $n$, the number of $\delta$-separated $t\in [0,1)$ such that 

\[\abs{\zeta_1\theta_1^t+\cdots+\zeta_r\theta_r^t-n}\leq \epsilon\]
is $O(\delta^{-1}\epsilon^{1/r}(\max \abs{\zeta_i})^{-1/r})$.

Let the set of such $t\in [0,1)$ be denoted by $T$. We can assume that $\abs{T}\geq 2^{r+2}$, say, or else we are done. It follows that there are $2^r$ many disjoint closed intervals $I_1,\ldots,I_{2^r}\subseteq [0,1)$ such that $T\subseteq \cup_i I_i$ and $\abs{T\cap I_i}\geq \lfloor\abs{T}/2^r\rfloor$ for $1\leq i\leq 2^r$. Let $f(t)=\sum_{j=1}^r \zeta_j \theta_j^t-n$ and for $1\leq i\leq 2^r$ let $v_i$ be a point in $I_i$ where $f'(t)$ is minimised. Note that, since all points in $T$ are $\delta$-separated, we have, for $1\leq i\leq 2^{r-1}$,
\[\abs{v_{2i+1}-v_{2i-1}}\geq \delta(\abs{T\cap I_{2i}}-1)\geq \frac{\delta}{2^{r+2}}\abs{T}.\]
By Lemma~\ref{lem-lb}, therefore, there exists some $1\leq i\leq 2^{r-1}$ such that
\[\min_{x\in I_{2i-1}}\abs{f'(x)}\geq \abs{f'(v_{2i-1})}\gg_{\theta_1,\ldots,\theta_r,r}\abs{T}^{r-1}\delta^{r-1}\max_i \abs{\zeta_i}.\]
We now use the general fact (which follows e.g. from the mean value theorem) that if $f:I\to \bbr$ has $\abs{f'(x)}\geq \lambda>0$ for all $x\in I$ and $t_1,\ldots,t_k\in I$ are $\delta$-separated points such that $\abs{f(t_i)}\leq \epsilon$ for $1\leq i\leq k$ then $\lambda \delta k\leq \epsilon$ to deduce that  
\[\abs{T}^r\delta^r\max_i \abs{\zeta_i}\ll \epsilon,\]
and the bound follows after rearranging this. 
\end{proof}

We will now deduce a more general form of this lemma, where the exponents are given by independent linear forms in some $t_1,\ldots,t_d\in [0,1)$. We need to be able to handle arbitrary linear forms in the exponents to handle the presence of possible linear dependencies between our $\frac{1}{\log_Lg_j}$. 

\begin{lemma}\label{lem:curve}
Let $1\leq d\leq r$. Let $\zeta_1,\ldots,\zeta_r\in \bbr$ and $\theta_1,\ldots,\theta_r> 1$ be such that $\theta_i^{a_i}\neq \theta_j^{a_j}$ for $i\neq j$ and any $a_i,a_j\in \bbq\backslash\{0\}$. Let $Q_{d+1},\ldots,Q_r$ be non-zero rational linear forms in $t_1,\ldots,t_d$. Suppose that
\[\abs{\sum_{1\leq i\leq d}\zeta_i\theta_i^{t_i}+\sum_{d< j\leq r}\zeta_j\theta_j^{Q_j(\vec{t})}}\leq Z\]
for all $t_1,\ldots,t_d\in [0,1)$. 

Let $\epsilon,\delta>0$. The number of $\delta$-separated points $\vec{t}\in [0,1)^d$ such that 
\[\norm{\sum_{1\leq i\leq d}\zeta_i\theta_i^{t_i}+\sum_{d< j\leq r}\zeta_j\theta_j^{Q_j(\vec{t})}}\leq \epsilon\]
is 
\[O_{r,\theta_1,\ldots,\theta_r,Q_{d+1},\ldots,Q_r}(Z\delta^{-d}(\max_i\abs{\zeta_i})^{-1/r}(\epsilon +\delta\max_i \abs{\zeta_i})^{1/r}).\]

\end{lemma}

\begin{proof}
Without loss of generality, we can assume that $\delta=1/D$ for some integer $D\geq 1$ and (since changing a $t_i$ by $\delta$ can affect the size of the sum by at most $O_{r,\theta_1,\ldots,\theta_r,Q_{d+1},\ldots,Q_r}(\delta \max_i\abs{\zeta_i})$) it suffices to show that, for any $\epsilon>0$, the number of $\vec{t}\in \{0,\ldots,D-1\}^d$ such that 
\begin{equation}\label{eq:curve1}
\norm{\sum_{1\leq i\leq d}\zeta_i(\theta_i^{1/D})^{t_i}+\sum_{d< j\leq r}\zeta_j(\theta_j^{1/D})^{Q_j(\vec{t})}}\leq \epsilon
\end{equation}
is
\[O_{r,\theta_1,\ldots,\theta_r,Q_{d+1},\ldots,Q_r}(ZD^d\epsilon^{1/r}(\max_i \abs{\zeta_i})^{-1/r}).\]

Let $2\leq M\ll 1$ be some integer to be chosen later. Any $\vec{t}\in \{0,\ldots,D-1\}^d$ can be written as 
\[(t_1,Mt_1+u_2,\ldots,M^{d-1}t_1+u_{d})\]
for some $u_2,\ldots,u_d\in \bbz\cap [-M^{d-1}D,D)$.

It follows that exists some $U\subset \bbz^{d}$ such that
\begin{enumerate}
\item $\abs{U}\ll D^{d-1}$,
\item $\Norm{\vec{u}}_\infty\ll D$ for all $\vec{u}\in U$, and 
\item if $\vec{t}\in \{0,\ldots,D-1\}^d$ then there exists some $\vec{u}\in U$ such that
\[\vec{t}=(t_1+u_1,Mt_1+u_2,\ldots,M^{d-1}t_1+u_{d}).\]
\end{enumerate}
It therefore suffices to show that, for any fixed $\vec{u}\in U$, the number of $t\in \{0,\ldots,D-1\}$ such that 
\[\norm{\sum_{1\leq i\leq d}\zeta_i\theta_i^{u_i/D}(\theta_i^{M^{i-1}})^{t/D}+\sum_{d< j\leq r}\zeta_j\theta_j^{Q_j(u_1,\ldots,u_d)/D}(\theta_j^{Q_j(1,M,\ldots,M^{d-1})})^{t/D}}\leq \epsilon\]
is
\[O_{r,\theta_1,\ldots,\theta_r,Q_{d+1},\ldots,Q_r}(ZD\epsilon^{1/r}(\max_i \abs{\zeta_i})^{-1/r}).\]
This follows by an application of Lemma~\ref{lem:curve1}, after choosing $M\ll 1$ such that all the numbers $\theta_1,\ldots,\theta_d^{M^{d-1}},\ldots,\theta_r^{Q_r(1,\ldots,M^{d-1})}$ are all distinct and $\neq 1$, and noting that 

\[\max(\Abs{\zeta_1}\theta_1^{u_1/D},\ldots,\Abs{\zeta_r}\theta_r^{Q_r(u_1,\ldots,u_d)/D})\gg_{r,\theta_1,\ldots,\theta_r,Q_{d+1},\ldots,Q_r} \max_i\abs{\zeta_i}.\]
\end{proof}

We can now prove the goal of this section, for which we will require the following form of Weyl's theorem.
\begin{lemma}[Multidimensional Weyl's theorem]\label{lem:weyl}
If $1,\theta_1,\ldots,\theta_d\in \bbr$ are linearly independent over $\bbq$ then the sequence
\[\brac{\{n\theta_1\},\ldots,\{n\theta_d\}}\]
is uniformly distributed in $[0,1)^d$ as $n\to\infty$.
\end{lemma}

\begin{lemma}\label{lem:badbound}
Let $g_1,\ldots,g_r\geq 2$ be such that there are no solutions to $g_{i}^{a_i}\neq g_{j}^{b_j}$ for all $i\neq j$ and integers $a_i,b_j\geq 1$. Let $\ell\geq 2$ satisfy $(\ell,g_1\cdots g_r)=1$ and let $L\geq \ell$ be a power of $\ell$. Let $\zeta_1,\ldots,\zeta_r\in \bbr$.

If $N$ is sufficiently large the number of $n\in [N]$ such that
\[\norm{\zeta_1g_1^{\{n/\log_Lg_1\}}+\cdots+\zeta_rg_r^{\{n/\log_Lg_r\}}}\leq \epsilon\]
 is 
\[\ll_{g_1,\ldots,g_r} (1+\max\abs{\zeta_i})(\max \abs{\zeta_i})^{-1/r}\epsilon^{1/r}N\]
\end{lemma}
\begin{proof}
Suppose that $\bbq(1/\log_\ell g_1,\ldots,1/\log_\ell g_r)$ has dimension $1\leq d\leq r+1$, as a vector space over $\bbq$. Without loss of generality, we can assume that we have some $b_{ji}\in \bbz$ and $q\geq 1$ such that, for $d\leq j\leq r$,
\begin{equation}\label{qlogg}
\frac{q}{\log_\ell g_{j}}=b_{j0}+\sum_{1\leq i< d}\frac{b_{ji}}{\log_\ell g_i},
\end{equation}
and that $1,1/\log_\ell g_1,\ldots,1/\log_\ell g_{d-1}$ are linearly independent over $\bbq$. Note in particular that $d\geq 2$, or else $\log_L g_1\in\bbq$, contradicting $(L,g_1)=1$. Similarly, the $b_{ji}$ are not identically zero. 

Let $L=\ell^k$, so that $\log_Lg_j = \tfrac{1}{k}\log_\ell g_j$. We have
\[\frac{1}{\log_Lg_j}=\frac{kb_{j0}}{q}+\sum_{1\leq i<d}\frac{b_{ji}}{q}\frac{1}{\log_Lg_i}.\]

Multiplying by $n$ and taking the fractional part it follows that, for any $n\geq 1$ and $d\leq j\leq r$,
\[\left\{\frac{n}{\log_Lg_{j}}\right\}=\frac{c_j(n)}{q}+\sum_{1\leq i<d}\frac{b_{ji}}{q}\left\{\frac{n}{\log_Lg_i}\right\},\]
for some $c_j(n)\in\bbz$. Since the left-hand side is $\in [0,1)$ we must have $\abs{c_j(n)}\ll_{\ell,g_1,\ldots,g_r}1$. 

In particular, if we write $t_i(n)=\{n/\log_Lg_i\}$ for $1\leq i<d$, then writing $Q_j(t_1,\ldots,t_{d-1})=\sum_{1\leq i<d}\frac{b_{ji}}{q}t_i$, 
\[\zeta_1g_1^{\{n/\log_Lg_1\}}+\cdots+\zeta_rg_r^{\{n/\log_Lg_r\}}=\sum_{1\leq i<d}\zeta_ig_i^{t_i(n)}+\sum_{d\leq j\leq r}\zeta_jg_j^{\tfrac{c_j(n)}{q}}g_j^{Q_j(\vec{t}(n))}.\]

In particular, since there are $O(1)$ many possible values for $c_j(n)$, there is some finite set $U\subset \bbr^{r}$ of size $O_{\ell,g_1,\ldots,g_r}(1)$ such that
\begin{enumerate}
\item $\abs{\zeta'_i}\asymp \abs{\zeta_i}$ for all $\vec{\zeta}'\in U$ and
\item if $\|\sum_i \zeta_ig_i^{\{n/\log_Lg_i\}}\|\leq \epsilon$ then there exists some $\vec{\zeta}'\in U$ such that 
\begin{equation}\label{eq:curve}
\norm{\sum_{1\leq i<d}\zeta_i'g_i^{t_i(n)}+\sum_{d\leq j\leq r}\zeta_j'g_j^{Q_j(\vec{t}(n))}}\leq \epsilon.
\end{equation}
\end{enumerate}
Since $1,1/\log_Lg_1,\ldots,1/\log_Lg_{d-1}$ are linearly independent over $\bbq$, by Lemma~\ref{lem:weyl} the sequence $\vec{t}(n)$ is uniformly distributed in $[0,1)^{d-1}$ as $n\to \infty$. In particular, for any $\delta>0$, provided $N$ is sufficiently large depending on $\delta$, the number of $n\in [N]$ such that $\vec{t}(n)$ lies in any fixed box in $[0,1)^{d-1}$ of width $\delta$ is $O(\delta^{d-1} N)$. 

By Lemma~\ref{lem:curve} there are 
\[\ll (1+\max\abs{\zeta_i})\delta^{1-d}\max\abs{\zeta_i}^{-1/r}(\epsilon+\delta \max_i \abs{\zeta_i})^{1/r}\]
many boxes that contain some point $\vec{t}$ satisfying \eqref{eq:curve} for some $\vec{\zeta}'\in U$, and the lemma follows, choosing $\delta=\epsilon/\max_i\abs{\zeta_i}$.
\end{proof}

\section{The large spectrum of integers with small digits}\label{sec-spec}
Finally, we will prove the spectral estimate that we require: namely that if $A$ is the set of integers in $[0,g^R)$ with small digits in some base then we can bound the number of integers $k\leq M$ such that the exponential sum $\sum_{n\in A}(nk/g^R)$ is large. 

For context, we first consider what the `trivial' bound is. Let
\[A=\left\{ \sum_{0\leq i<R}c_ig^{i} : 0\leq c_i< g/2\textrm{ for }0\leq i< R\right\}\]
and
\[\Delta(M) = \{ k\leq M : \lvert \sum_{n\in A}e(nk/g^R)\rvert\geq M^{-\delta}\lvert A\rvert\}.\]
By Parseval's identity we have, when $M=g^R$,
\[\sum_{k\leq M}\abs{\sum_{n\in A}e(nk/g^R)}^2 \ll \abs{A}M\]
and hence, since $\abs{A}\gg (g/2)^R=M^{1-\log_g(2)}$,
\[\abs{\Delta(M)} \ll M^{1+2\delta}\abs{A}^{-1}\ll  M^{\log_g(2)+2\delta}.\]
As $\delta\to 0$ and $g\to \infty$ this is, of course, a great improvement over the very trivial bound of $\abs{\Delta(M)}\ll M$. This only works, however, if $M\asymp g^R$. 

We require stronger information however, and in particular for our application need an effective bound even when $M=(g^R)^{o(1)}$. In such a regime the bound arising from Parseval's identity becomes worse than trivial. 

We will prove a bound that implies, in this situation, for any $M\ll g^R$, 
\[\abs{\Delta(M)}\ll M^{\log_g(20)+2\sqrt{\delta}}.\]
The exponent here is slightly worse than what Parseval's identity delivers when $M$ is comparable to $g^R$, but the key point is that we maintain a power type saving (for sufficiently large $g$) even when $M=(g^R)^{o(1)}$. 

The key observation is that, since the exponential sum over $A$ can be factored as the product of $R$ short exponential sums, each over a short interval, the size of the exponential sum is essentially governed by the base $g$ expansion of $k$ itself. In particular the exponential sum can only be large if many of the base $g$ digits of $k$ are large, which can only happen $M^{o(1)}$ many times in $[1,M]$. 

It is worth noting the duality here: the large spectrum of a set of integers with small digit requirements is itself a set of integers with small digit requirements. This is a generally useful heuristic; it is essentially the same phenomenon as the Fourier transform of a Gaussian being another Gaussian.

\begin{lemma}\label{lem:spec}
Let $R\geq 1$ and $\kappa\in(0,1)$. Let $g >t\geq 1$ be integers. If 
\[A=\left\{ \sum_{0\leq i<R}c_ig^{i} : 0\leq c_i< t\textrm{ for }0\leq i< R\right\}\]
then, for any $K\geq 1$,
the number of integers $0\leq k<g^K$ such that
\begin{equation}\label{eq-spec}
\abs{\sum_{n\in A}e(nk/g^R)}\geq \eta \abs{A}
\end{equation}
is at most
\[(10/t)^{\min(R,K)}\exp\brac{2\sqrt{K\log t\log(1/\eta)}}g^K\]
In particular, if $X\geq 2$, for any $1\leq M\leq Xg^R$ and $\delta>0$, the number of integers $0\leq k<M$ such that
\[\abs{\sum_{n\in A}e(nk/g^R)}\geq M^{-\delta}\abs{A}\]
is at most 
\[X^{O_g(1)}M^{\log_g(10g/t)+2\sqrt{\delta}}.\]
\end{lemma}
\begin{proof}
Note that the conclusion is trivial if $t=1$, and so we may henceforth assume that $t\geq 2$. Let $\Delta=\{ 0\leq k<g^K : \lvert \sum_{n\in A}e(nk/g^R)\rvert \geq \eta\abs{A}\}$. We have
\[\sum_{n\in A}e(nk/g^R)=\prod_{0\leq i<R}\brac{\sum_{0\leq c<t}e(ckg^{i-R})}.\]
Summing the geometric series yields, for any $\theta\in \bbr$, 
\[\abs{\sum_{0\leq c<t}e(c\theta)}\leq \min(t,\norm{\theta}^{-1}).\]
Suppose $k = \sum_{j\geq 0}a_j g^j$ with $0\leq a_j<g$. Note that, for $0\leq i<R$, 
\[\norm{kg^{i-R}}=\norm{\sum_{0\leq j<R-i}a_jg^{j+i-R}}=\norm{\frac{a_{R-i-1}}{g}}+\delta_i,\]
say, for some $\abs{\delta_i}< 1/g$. Writing $\tilde{a}$ for the integer with minimal  $\abs{\tilde{a}}$ such that $\tilde{a}\equiv a\pmod{g}$, it follows that

\[\norm{k g^{j-R}}>\frac{\abs{\tilde{a}_{R-j-1}}-1}{g}.\]
Let $C>1$ be some parameter to be chosen later and let $J$ be the set of those $0\leq i<R$ such that $\Abs{\tilde{a}_{R-i-1}}\geq 2Cg/t$. For $j\in J$ we have 
\[\norm{k g^{j-R}}>\frac{\abs{\tilde{a}_{R-j-1}}}{2g}\geq C/t\]
and hence
\[\abs{\sum_{0\leq c<t}e(ckg^{j-R})}\leq t/C.\]
For all $j\not\in J$ we will bound this exponential sum trivially by $t$. It follows that (as $\abs{A}=t^R$)
\[
\abs{\sum_{n\in A}e(nk/g^R)}\leq C^{-\abs{J}} \abs{A},\]
and hence for the left-hand side to be $\geq \eta\abs{A}$ we must have $\abs{J}\leq \log_C(1/\eta)$. 

Let $K'=\min(R,K)$. We now note that the number of $k<g^K$ where all except $r$ of the digits from places $0\leq i<R$ are inside some set $\subseteq \{0,\ldots,g-1\}$ of size $T$ is at most
\[\binom{K'}{r}g^{K-K'+r}T^{K'-r}.\]
Applying this with $T=\abs{[-2Cg/t,2Cg/t]}\leq 5Cg/t$ and $r=\lfloor \log_C(1/\eta)\rfloor$, it follows that 
\[\abs{\Delta}\leq 2^{K'}g^K(5C/t)^{K'-r}\leq (10C/t)^{K'}t^{r}g^K\leq (10/t)^{K'}C^Kt^{\frac{\log(1/\eta)}{\log C}}g^K.\]
In particular if we choose $C>1$ such that $(\log C)^2=\frac{\log t\log(1/\eta)}{K}$ then
\[\abs{\Delta}\leq (10/t)^{K'}\exp\brac{2\sqrt{K\log t\log(1/\eta)}}g^K\]
as required. To deduce the second part, it remains only to note that without loss of generality we can assume that $M=g^K\leq Xg^R$, whence
\[(10/t)^{K'}g^K\leq X^{O(1)}M^{\log_g(10g/t)}\]
and
\[\exp\brac{2\sqrt{K\log t\log(1/\eta)}}=M^{2\sqrt{\delta\log_gt}}\leq M^{2\sqrt{\delta}}.\]
\end{proof}

This concludes the proof of Theorem~\ref{th-ologn}.
\section{Quantitative improvements}\label{sec-imp}

As already mentioned, it would be desirable to have a more quantitatively effective version of Theorem~\ref{th-ologn}. If one examines the proof, the only source of ineffectivity is the purely qualitative `for sufficiently large $N$' in Proposition~\ref{prop-main}, which in turn arises from the use of Weyl's criterion (Lemma \ref{lem:weyl}). We recall that this is used at a critical stage in the proof of Lemma \ref{lem:badbound}, where one has some $1, 1/ \log_L g_1, 1/ \log_L g_2, ..., 1/\log_L g_{d-1}$ linearly independent over ${\mathbb Q}$, and then uses Weyl's critertion to deduce that as $n\to \infty$ the vector $(\{n/\log_L g_1\}, ..., \{n/\log_L g_{d-1}\})$ is uniformly distributed in $[0,1)^{d-1}$.

To make our final result effective we need a quantified form of Weyl's equidistribution theorem, such as the Erd\H{o}s-Tur\'{a}n-Koksma inequality \cite{koksma}.  This would mean that, instead of using the fact that
${\mathbb Q}(1/\log_L g_1, ..., 1/\log_L g_r)$ has dimension $1 \leq d \leq r+1$ over ${\mathbb Q}$, we would use an approximate version of this.  Specifically, would need to assume that (after some reordering) there exists some $1\leq d\leq r+1$ such that there is no short integer linear combination of $1, 1/\log_L g_1,..., 1/\log_L g_{d-1}$ (i.e. linear combinations 
$m_1 + m_2 /\log_L g_1 + \cdots + m_d / \log_L g_{d-1}$, 
where the $m_i$ satisfy some bound $|m_i| < M$, not all $0$, for an appropriate choice of $M$) too close to $0$ in magnitude; furthermore, each of the remaining $g_i$ will have the property that an approximate version of (\ref{qlogg}) holds.  

There are a number of technical issues that present when attempting to implement this precisely, but heuristically this should give a version of Lemma \ref{lem:badbound} where the conclusion to that lemma now holds for $N\ll \epsilon^{-O(1)}$. Since $N\asymp \log_L n$ and $\log L\asymp \epsilon^{-1}\log(1/\epsilon)$ in the proof of Theorem~\ref{th-ologn} this in turn means that we may choose
\[\log n \asymp \epsilon^{-O(1)}.\]
Choosing $\epsilon>0$ to satisfy this, and proceeding with the rest of the proof of Theorem~\ref{th-ologn} unchanged, would prove a quantitatively effective version of Theorem~\ref{th-ologn} in which one could replace $\epsilon \log n$ by $(\log n)^{1-c}$ for some constant $c>0$ depending on the $g_i$ and $\kappa_j$.

A more direct (albeit much more difficult) way that we could get effective bounds in Theorem \ref{th-ologn} would be to prove a quantitative version of a special case of Schanuel's conjecture \cite{lang}.

\begin{conjecture}[Schanuel's conjecture]\label{conj-schanuel}
Given any $z_1,...,z_r\in\bbc$ that are linearly independent over $\bbq$ the field extension ${\mathbb Q}(z_1,\ldots,z_r,e^{z_1},\ldots,e^{z_r})$ has transcendence degree at least $r$ over ${\mathbb Q}$.
\end{conjecture}

Taking $z_1 = \log g_1$, $z_2 = \log g_2$, ..., 
$z_r = \log g_r$, we see that one consequence of this conjecture is that if
$$
{m_1 \over \log_L g_1} + \cdots + {m_r \over \log_L g_r}\ \in\ {\mathbb Z},
$$
with $m_1,...,m_r \in {\mathbb Z}$, then 
$m_1=\cdots = m_r = 0$.  This would simplify the proof of Lemma \ref{lem:badbound}
somewhat and give a stronger conclusion. To obtain a strong, effective bound in Theorem \ref{th-ologn}, without working with approximate integer linear dependencies as described at the beginning of this subsection, then we would need to use the Erd\H{o}s-Tur\'{a}n-Koksma theorem (as a quantitative version of Weyl's lemma) and a quantitative statement like the following.

\begin{conjecture}\label{strongschanuel} Let $g_1,...,g_r \geq 2$ be integers
such that $g_i^{a_i} \neq g_j^{b_j}$ for all $i \neq j$ and integers $a_i, b_j \geq 1$.  Then, for all $M \geq 2$,
$$
\min_{m_1,...,m_r} \left \| {m_1 \over \log_L g_1} + \cdots + {m_r \over 
\log_L g_r} \right \|\ >\ {1 \over M^{r (1+o(1))}},
$$
where the minimum is taken over all integers $m_1,..., m_r$ satisfying $|m_i| \leq M$, $i=1,...,r$, and not all $m_i = 0$.
\end{conjecture}

This conjecture is based on the following heuristic. Assuming the $1/\log_L g_i$ mod $1$ behave like random real numbers in the interval $(0,1)$, the set of integer linear combinations $m_1 / \log_L g_1 + \cdots +m_r /\log_L g_r$ mod $1$ vary over $(2M+1)^r-1$ numbers as the $m_i$ vary over the integers in $[-M,M]$, not all $0$.  If these linear combinations were themselves like random real numbers in $(0,1)$, we would expect none of them come closer than about $(2M+1)^{-r}$ to $0$. 

If one traces the consequences of this conjecture through the proof in this paper, it gives infinitely many integers $n$ such that for all $i=1,...,r$, all but at most $(\log n)^{1/2 + o(1)}$ base $g_i$ digits of $n$ are $< \kappa_i g_i$. All the implied constants would be effectively computable, provided the implied constants in Conjecture \ref{strongschanuel} are computable.

\appendix

\section{Construction of bump function}\label{app}
In this appendix we give details of the construction of the bump function used in the proof of Proposition~\ref{prop-fourier}. Of course, bump functions with these properties are ubiquitous in Fourier analysis, but since we could not find an easy reference for the precise properties we required, we have elected to give a proof here for the benefit of the reader.
\begin{lemma}
For any $\delta\in(0,1)$ and $J\geq 1$ there exists a function $\psi:\bbr/\bbz\to \bbr$ supported on $[-\delta/2,\delta/2]$ such that
\begin{enumerate}
\item $\widehat{\psi}(0)=\| \psi\|_1=1$,
\item $0\leq \widehat{\psi}(k)\leq \min(1,(J^2/\delta k)^{2J})$ for all $k\in \mathbb{Z}$, and
\item $\sum_{k\in \bbz} \widehat{\psi}(k)=\psi(0)\leq 4/\delta$.
\end{enumerate}

\end{lemma}
\begin{proof}
Let $c_1,\ldots,c_J\geq 0$ be constants to be chosen later with $\sum c_j\leq 1/2$ and $\psi_j=(c_j\delta)^{-1}1_{[-c_j\delta/2,c_j\delta/2]}$. The Fourier transform of $\psi_j$ is
\[\widehat{\psi_j}(k)=\frac{\sin(\pi c_j\delta k)}{\pi c_j \delta k}.\]
We choose $\psi=\psi_1\ast \psi_1\ast \cdots \ast \psi_J\ast \psi_J$. Note that $\psi\geq 0$ pointwise and therefore
\[\widehat{\psi}(0)=\int_0^1 \psi(t)\td t=\norm{\psi}_1=\prod_{j} \norm{\psi_j}_1^2=1,\]
since each $\psi_j$ was normalised to have $\norm{\psi_j}_1=1$. Furthermore since $\sum_j c_j\leq 1/2$ this is supported on $[-\delta/2,\delta/2]$ as required, and 
\[\widehat{\psi}=\prod_j \widehat{\psi_j}^2\]
and so in particular $\widehat{\psi}\geq 0$ pointwise. Furthermore
\[\sum_k \widehat{\psi}(k)\leq \sum_k \widehat{\psi_1}^2= \| \psi_1\|_2^2=(c_1\delta)^{-1}.\] 
We will choose $c_1=1/4$. It remains to show that
\[\widehat{\psi}(k)\leq (J^2/\delta k)^{2J}\]
for all $k\in \bbz$. The left-hand side is
\[(\pi \delta k)^{-2J}\prod_{j}\frac{\sin(\pi c_j\delta k)^2}{c_j^2}.\]
If we choose $c_j=1/4j^2$ then this right-hand side is 
\[\leq 4^{2J}(\pi \delta k)^{-2J}(J!)^{4}.\]
Using $J! \leq J^J/e^J$ and $4/\pi e^2<1$ this is at most $(J^2/\delta k)^{2J}$ as required.
\end{proof}

\begin{thebibliography}{0}

\bibitem{banks} B. Banks, A. Conflitti, and I. Shparlinski, {\it Characters sums over integers with restricted g-ary digits}, Illinois J. of Math {\bf 46} (2002), 819-836.

\bibitem{CMS}
E. Croot, H. Mousavi, and M. Schmidt.
\textit{On a conjecture of Graham on the $p$-divisibility of central binomial coefficients}.

\bibitem{EG}
P. Erd\H{o}s and R. L. Graham.
\textit{Old and new problems and results in combinatorial number theory}, 1980.

\bibitem{EGRS}
P. Erd\H{o}s, R. L. Graham, I. Z. Ruzsa, and E. G. Straus.
\textit{On the prime factors of $\binom{2n}{n}$}, Math. Comp. (29) 1975, p. 83-92.

\bibitem{ford} K. Ford and S. Konyagin, {\it Divisibility of the central binomial coefficients ${2n \choose n}$}, Trans. Amer. Math. Soc. {\bf 374} (2021), 923-953.

\bibitem{FoMa05} 
E. Fouvry and C. Mauduit, {\it Sur les entiers dont la somme des chiffres est moyenne}, Journal of Number Theory \textbf{114} (2005), 135-152.

\bibitem{granville} A. Granville, {\it Missing digits and good approximations}, Bull. Amer. Math. Soc. {\bf 61} (2024), 23-53.

\bibitem{koksma} J. F. Koksma, \textit{Some theorems on diophantine inequalities}, Scriptum, vol. 5, Math. Centrum Amsterdam, 1950.

\bibitem{lang} S. Lang, {\it Introduction To Transcendental Numbers},
Addison-Wesley, 1966, pp. 30-31.

\bibitem{ma} J. Maynard, {\it Primes with restricted digits}, Invent. Math. {\bf 217} (2019), 127–218.

\bibitem{Po}
C. Pomerance.
\textit{Divisors of the middle binomial coefficient},
American Mathematical Monthly, 2014.

\end{thebibliography}
\end{document}